\documentclass[11pt]{article}
\usepackage{amsmath, amsthm, amssymb, geometry, enumitem}
\usepackage{hyperref}
\usepackage{mathtools}
\usepackage{authblk}
\usepackage{orcidlink}

\title{Fixed Point Theorem for Path-Averaged Contractions in Complete b-Metric Spaces}
\author{Nicola Fabiano\, \orcidlink{0000-0003-1645-2071}}
\affil{``Vin\v{c}a'' Institute of Nuclear Sciences - National 
Institute of the Republic of Serbia, University of Belgrade, Mike Petrovi\'{c}a 
Alasa 12--14, 11351 Belgrade, Serbia; nicola.fabiano@gmail.com}
\date{}

\newtheorem{theorem}{Theorem}
\newtheorem{lemma}[theorem]{Lemma}
\newtheorem{proposition}[theorem]{Proposition}
\newtheorem{definition}[theorem]{Definition}
\newtheorem{example}[theorem]{Example}
\newtheorem{remark}[theorem]{Remark}
\newtheorem{corollary}[theorem]{Corollary}
\begin{document}

\maketitle


\begin{abstract}
We extend the fixed point result for Path-Averaged Contractions (PA-contractions) from complete metric spaces to complete b-metric spaces. We prove that every  PA-contraction on a complete b-metric space has a unique fixed point, provided the contraction constant $ \alpha $ satisfies $s \alpha^{1/N} < 1$, where $ s \geq 1 $ is the b-metric coefficient and $N$ the averaging parameter. 
Moreover, we establish that every PA-contraction is automatically continuous.
The proof relies on geometric decay of successive distances and the generalized triangle inequality. 
This result paves the way for extending averaged contraction principles to other classical types, such as Kannan, Chatterjea, and Ćirić-type mappings, as well as Wardowski’s F-contractions, in generalized metric settings.
\end{abstract}

\section{Introduction}

Path-Averaged Contractions (PA-contractions) were recently introduced as a generalization of Banach contractions, where contraction is measured in terms of the average distance along orbits~\cite{Fpathaveraged}. It has been established that such mappings possess a unique fixed point in complete metric spaces. In this paper, we extend this result to the broader class of \textbf{complete b-metric spaces}, which relax the standard triangle inequality by allowing a multiplicative constant $ s \geq 1 $.

We show that the fixed point property persists under the additional condition that  $s \alpha^{1/N} < 1$, and provide a self-contained proof of existence and uniqueness. These assumptions arise naturally from the structure of b-metric spaces and may be removable in symmetric subclasses of contractions.

Our main result establishes the existence and uniqueness of fixed points for PA-contractions in b-metric spaces. We compare PA-contractions with two foundational classes: Banach and Kannan contractions. We prove that every Banach contraction is a PA-contraction, but not conversely; hence, PA-contractions strictly generalize Banach's principle. For Kannan contractions, we show that there exist 
PA-contractions that are not Kannan contractions. The question for the converse relationship remains open.

The investigation of PA-type conditions for other classical principles — including Chatterjea \cite{Chatterjea1972}, \'Ciri\'c-type mappings \cite{Ciric1974}, and Wardowski’s F-contractions \cite{Wardowski2012,fabiano2022} — remains open and will be addressed in subsequent work.
\section{Preliminaries}
\label{sec:preliminaries}
We begin with the definition of a b-metric space.

\begin{definition}[b-metric space]
Let $ X $ be a nonempty set, $ s \geq 1 $ a real number, and $ d: X \times X \to [0, \infty) $ a function satisfying:
\begin{enumerate}[label=\textup{(\roman*)}]
    \item $ d(x, y) = 0 $ if and only if $ x = y $,
    \item $ d(x, y) = d(y, x) $ for all $ x, y \in X $,
    \item $ d(x, z) \leq s \big( d(x, y) + d(y, z) \big) $ for all $ x, y, z \in X $.
\end{enumerate}
Then $ (X, d) $ is called a \textbf{b-metric space} with coefficient 
$ s $ \cite{Czerwik1993}. 
A sequence $ \{x_n\} $ in $ X $ is a Cauchy sequence if $ \lim_{m,n \to \infty} d(x_n, x_m) = 0 $. The space is \textbf{complete} if every Cauchy sequence converges in $ X $.
\label{bmetricspace}
\end{definition}

We now define the main concept of this paper.

\begin{definition}[PA-contraction]
Let $ (X, d) $ be a b-metric space. A mapping $ T: X \to X $ is called a \textbf{PA-contraction} (Path-Averaged Contraction) if there exist constants $ \alpha \in (0,1) $ and $ N \in \mathbb{N} $ such that for all $ x, y \in X $ and all $ n \geq N $,
\begin{equation}
    \sum_{k=0}^{n-1} d(T^{k+1}x, T^{k+1}y) \leq \alpha \sum_{k=0}^{n-1} d(T^k x, T^k y).
    \label{eq:pa}
\end{equation}
\end{definition}

This condition implies that the cumulative distance between iterated images contracts on average by a uniform factor $ \alpha < 1 $, independent of the initial points $ x, y $. Such path-averaging ideas have parallels in orbital continuity \cite{pantorbit,Rhoades1977} and generalized contractions \cite{Proinov2020}.

\section{Main Result}
\label{sec:mainresult-pa}



In the following lemma, we address an important property of PA-contractions: prove that every PA-contraction is automatically continuous.

\begin{lemma}[Continuity of PA-Contractions]
\label{thm:continuity}
Let $(X, d)$ be a b-metric space with coefficient $s \geq 1$, and let $T: X \to X$ be a PA-contraction with constants $\alpha \in (0,1)$ and $N \in \mathbb{N}$. Then $T$ is Lipschitz continuous.
\end{lemma}

\begin{proof}
Let $T$ be a PA-contraction with constants $\alpha \in (0,1)$ and $N \in \mathbb{N}$. We will show that $T$ satisfies a Lipschitz condition.

Fix $x, y \in X$. Since $T$ is a PA-contraction, for $n = N$ we have
\[
\sum_{k=0}^{N-1} d(T^{k+1}x, T^{k+1}y) \leq \alpha \sum_{k=0}^{N-1} d(T^k x, T^k y).
\]

Define $D_k = d(T^k x, T^k y)$ for $k = 0, 1, \dots, N$. Then the inequality becomes
\[
D_1 + D_2 + \cdots + D_N \leq \alpha (D_0 + D_1 + \cdots + D_{N-1}).
\]

Bringing all terms to the left side
\[
(1-\alpha)D_1 + (1-\alpha)D_2 + \cdots + (1-\alpha)D_{N-1} + D_N \leq \alpha D_0.
\]

Since all terms on the left are nonnegative, we have in particular
\[
(1-\alpha)D_1 \leq \alpha D_0,
\]
as the inequality must hold for each nonnegative term individually.

This implies
\[
d(Tx, Ty) = D_1 \leq \frac{\alpha}{1-\alpha} d(x, y).
\]

Therefore, $T$ is Lipschitz continuous with Lipschitz constant 
$L = \frac{\alpha}{1-\alpha}$.
\end{proof}

\begin{remark}
This result shows that PA-contractions inherit the continuity property from Banach contractions. The Lipschitz constant $\frac{\alpha}{1-\alpha}$ depends only on the contraction parameter $\alpha$ and is independent of the b-metric coefficient $s$ and the averaging parameter $N$.
\end{remark}



Our goal now is to prove the following theorem.


\begin{theorem}[Main Theorem]
\label{thm:main}
Let $(X, d)$ be a complete b-metric space with coefficient $s \geq 1$, and let $T: X \to X$ be a PA-contraction with constants $\alpha \in (0,1)$ and $N \in \mathbb{N}$ such that $s \alpha^{1/N} < 1$. Then $T$ has a unique fixed point in $X$, and for every $x \in X$, the sequence $\{T^n x\}$ converges to this fixed point.
\end{theorem}


The proof proceeds in several steps: we first show that the sequence of iterates is a Cauchy sequence, then that it converges to a fixed point, and finally that the fixed point is unique.

\subsection{Step 1: Summability of successive distances}

\begin{proof}

Fix an arbitrary point $ x \in X $, and define the sequence $ x_n = T^n x $. Let
\[
a_k = d(T^k x, T^{k+1} x) = d(x_k, x_{k+1}).
\]
We will show that the series $ \sum_{k=0}^\infty a_k $ converges.

Consider the pair $ (x, Tx) $. Apply the PA-contraction condition to this pair. For $ n \geq N $, we have
\[
\sum_{k=0}^{n-1} d(T^{k+1}x, T^{k+1}(Tx)) \leq \alpha \sum_{k=0}^{n-1} d(T^k x, T^k (Tx)).
\]
But $ T^k(Tx) = T^{k+1}x $, so
\[
d(T^k x, T^k (Tx)) = d(T^k x, T^{k+1} x) = a_k,
\]
and
\[
d(T^{k+1}x, T^{k+1}(Tx)) = d(T^{k+1}x, T^{k+2}x) = a_{k+1}.
\]
Thus, the inequality becomes
\[
\sum_{k=0}^{n-1} a_{k+1} \leq \alpha \sum_{k=0}^{n-1} a_k
\quad \Rightarrow \quad
\sum_{k=1}^{n} a_k \leq \alpha \sum_{k=0}^{n-1} a_k.
\]

Let $ S_n = \sum_{k=0}^{n-1} a_k $. Then
\[
\sum_{k=1}^{n} a_k = \sum_{k=0}^{n} a_k - a_0 = S_{n+1} - a_0,
\]
and the inequality $ \sum_{k=1}^{n} a_k \leq \alpha S_n $ becomes
\[
S_{n+1} - a_0 \leq \alpha S_n
\quad \Rightarrow \quad
S_{n+1} \leq \alpha S_n + a_0.
\]
This is a recursive inequality. Since $ \alpha \in (0,1) $, standard analysis shows that $ \{S_n\} $ is bounded. Indeed, iterating
\begin{align*}
S_{n+1} &\leq \alpha S_n + a_0 \\
&\leq \alpha (\alpha S_{n-1} + a_0) + a_0 = \alpha^2 S_{n-1} + \alpha a_0 + a_0 \\
&\leq \cdots \leq \alpha^n S_1 + a_0 (1 + \alpha + \cdots + \alpha^{n-1}) \\
&\leq \alpha^n S_1 + \frac{a_0}{1 - \alpha}.
\end{align*}
Hence, $ \sup_n S_n < \infty $, so $ \sum_{k=0}^\infty a_k < \infty $. In particular, $ a_k \to 0 $.

\subsection{Step 2: The sequence $ \{T^n x\} $ is a Cauchy sequence}

We now prove that the sequence $ \{a_k\} $ decays geometrically. 
Fix $ p \geq 0 $ and apply the PA-contraction condition to the pair $ (T^p x, T^{p+1} x) $. 
For all $ n \geq N $, we have
\[
\sum_{k=0}^{n-1} a_{p+k+1} \leq \alpha \sum_{k=0}^{n-1} a_{p+k}.
\]
Let $ n = N $. Then
\[
\sum_{k=0}^{N-1} a_{p+k+1} \leq \alpha \sum_{k=0}^{N-1} a_{p+k}.
\]
This is equivalent to
\[
\sum_{k=1}^{N} a_{p+k} \leq \alpha \left( a_p + \sum_{k=1}^{N-1} a_{p+k} \right).
\]
Rearranging, we obtain
\[
a_{p+N} + (1-\alpha) \sum_{k=1}^{N-1} a_{p+k} \leq \alpha a_p.
\]
Since all terms are nonnegative, we have
\[
a_{p+N} \leq \alpha a_p.
\]
By induction, for any integer $ m \geq 0 $,
\[
a_{p+mN} \leq \alpha^m a_p.
\]
Now, for any $ k \in \mathbb{N} $, write $ k = mN + r $ with $ 0 \leq r < N $. Then
\[
a_k = a_{mN+r} \leq \alpha^m a_r.
\]
Let $ M = \max\{a_0, a_1, \dots, a_{N-1}\} $. Then
\[
a_k \leq M \alpha^m = M \alpha^{(k-r)/N} = \frac{M}{\alpha^{r/N}} \alpha^{k/N} \leq \frac{M}{\alpha^{(N-1)/N}} \alpha^{k/N}.
\]
Define $ C = \frac{M}{\alpha^{(N-1)/N}} $ and $ \gamma = \alpha^{1/N} $. Then $ \gamma < 1 $ and
\[
a_k \leq C \gamma^k.
\]

Now we estimate $ d(T^n x, T^m x) $ for $ m > n $. By the b-metric triangle inequality:
\[
d(T^n x, T^m x) \leq \sum_{j=n}^{m-1} s^{j-n+1} a_j \leq \sum_{j=n}^{m-1} s^{j-n+1} C \gamma^j = C s \gamma^n \sum_{i=0}^{m-n-1} (s \gamma)^i.
\]
Since $ s \gamma = s \alpha^{1/N} < 1 $, the geometric series converges:
\[
d(T^n x, T^m x) \leq \frac{C s \gamma^n}{1 - s \gamma}.
\]
As $ n \to \infty $, the right-hand side tends to 0, so $ \{T^n x\} $ is a Cauchy sequence.

Since $ X $ is complete, there exists $ z \in X $ such that $ T^n x \to z $ as $ n \to \infty $.


\subsection{Step 3: $ z $ is a fixed point of $ T $}

We now show that $ Tz = z $.
Since $ T^n x \to z $ and $ T $ is continuous (by Lemma~\ref{thm:continuity}), it follows that
\[
T(T^n x) = T^{n+1} x \to Tz \quad \text{as } n \to \infty.
\]
On the other hand, $ \{T^{n+1} x\} $ is a subsequence of $ \{T^n x\} $, so $ T^{n+1} x \to z $. In a b-metric space, limits are unique: if $ x_n \to x $ and $ x_n \to y $, then
\[
d(x, y) \leq s \big( d(x, x_n) + d(x_n, y) \big) \to 0,
\]
so $ d(x, y) = 0 $ and hence $ x = y $. Therefore, $ Tz = z $.

Thus, $ z $ is a fixed point of $ T $.


\subsection{Step 4: Uniqueness of the fixed point}

Suppose $ z, w \in X $ are two fixed points of $ T $, i.e., $ Tz = z $, $ Tw = w $. Then for all $ k \geq 0 $, $ T^k z = z $, $ T^k w = w $. Thus,
\[
d(T^k z, T^k w) = d(z, w).
\]
Apply the PA-contraction condition to $ (z, w) $ for $ n \geq N $
\[
\sum_{k=0}^{n-1} d(T^{k+1}z, T^{k+1}w) \leq \alpha \sum_{k=0}^{n-1} d(T^k z, T^k w)
\quad \Rightarrow \quad
\sum_{k=0}^{n-1} d(z, w) \leq \alpha \sum_{k=0}^{n-1} d(z, w).
\]
This simplifies to
\[
n \cdot d(z, w) \leq \alpha \cdot n \cdot d(z, w)
\quad \Rightarrow \quad
(1 - \alpha) d(z, w) \leq 0.
\]
Since $ \alpha < 1 $, we have $ 1 - \alpha > 0 $, so $ d(z, w) = 0 $. Hence, $ z = w $.

\end{proof}

\begin{remark}
The condition $s\alpha^{1/N} < 1$ arises from the geometric decay rate 
$\gamma = \alpha^{1/N}$ established in the proof and the requirement that 
the series $\sum (s\gamma)^k$ converges. This is stronger than the 
condition $\alpha < 1$ that suffices for some Banach contraction results 
in b-metric spaces, but is necessary for the PA-contraction framework 
due to the path-averaged nature of the contraction condition.
\end{remark}


\section{Banach contraction in b-metric space}
\label{sec:banach}

Recently, Path-Averaged (PA) contractions have been introduced, where contraction is measured in terms of the average distance along orbits.

In the following sections, we clarify the relationship between the notion of Path-Averaged contractions and Banach contraction in metric spaces. Specifically, we prove:
\begin{itemize}
\item Every Banach contraction is a PA-contraction.
    \item Not every PA-contraction is a Banach contraction.
\end{itemize}
We provide a finite, explicit counterexample using the discrete metric to illustrate the proper inclusion.

We recall that
a mapping $ T: X \to X $ is a \textbf{Banach contraction} if there exists $ \beta \in (0,1) $ such that
\[
d(Tx, Ty) \leq \beta\, d(x, y), \quad \text{ for all } x, y \in X.
\]

\label{sec:ba-pa}
We first show that the class of Banach contractions is contained in the class of PA-contractions.

\begin{proposition}
Let $ (X, d) $ be a metric space. If $ T: X \to X $ is a Banach contraction with constant $ \beta \in (0,1) $, then $ T $ is a PA-contraction with $ \alpha = \beta $ and $ N = 1 $.
\end{proposition}

\begin{proof}
Immediate.
\end{proof}

\label{sec:pa-not-ba}

We now show that the converse is false: there exist PA-contractions that are not Banach contractions.

\begin{example}
\label{ex:discretemetric}
Let $ X = \{0, 1, 2\} $, and define the \textbf{discrete metric}
\[
d(x, y) = 
\begin{cases}
0 & \text{if } x = y, \\
1 & \text{if } x \ne y.
\end{cases}
\]
Define $ T: X \to X $ by
\[
T(0) = 1, \quad T(1) = 2, \quad T(2) = 2.
\]
Then
\begin{enumerate}[label=\textup{(\alph*)}]
    \item $ T $ is \textbf{not} a Banach contraction.
    \item $ T $ \textbf{is} a PA-contraction.
\end{enumerate}
\end{example}

\begin{proof}
We prove each part separately.

\subsection*{(a) $ T $ is not a Banach contraction}

Suppose, for contradiction, that $ T $ is a Banach contraction with some $ \beta \in (0,1) $. Then
\[
d(Tx, Ty) \leq \beta\, d(x, y), \quad \forall x, y \in X.
\]

Consider $ x = 0 $, $ y = 1 $. Then,
$ d(0,1) = 1 $;
$ d(T(0), T(1)) = d(1,2) = 1 $.

So $ 1 = d(T(0), T(1)) \leq \beta \cdot 1 = \beta < 1 $, a contradiction.

Similarly, $ d(T(0), T(2)) = d(1,2) = 1 = d(0,2) $, so again $ 1 \leq \beta < 1 $ follows.

Hence, $ T $ is not a Banach contraction.

\subsection*{(b) $ T $ is a PA-contraction}

We show that there exist $ \alpha \in (0,1) $, $ N \in \mathbb{N} $ such that for all $ x, y \in X $ and all $ n \geq N $,
\[
\sum_{k=0}^{n-1} d(T^{k+1}x, T^{k+1}y) \leq \alpha \sum_{k=0}^{n-1} d(T^k x, T^k y).
\]

First, compute the iterates of $ T $
\begin{align*}
T^0(x) &= x, \\
T^1(x) &= 1 \text{ for } x=0, \quad 2 \text{ for } x=1,2, \\
T^2(0) &= T(1) = 2, \quad T^2(1) = T(2) = 2, \quad T^2(2) = 2, \\
T^k(x) &= 2 \quad \text{for all } k \geq 2, \text{ and all } x \in X.
\end{align*}

Thus, all orbits reach the fixed point $ 2 $ by the second iteration and remain there.

Now, analyze all distinct pairs $ (x, y) $, $ x \ne y $. There are three such unordered pairs.

\subsubsection*{Case 1: $ x = 0, y = 1 $}

Compute $ \delta_k = d(T^k(0), T^k(1)) $:
\item $ k=0 $: $ d(0,1) = 1 $,
    \item $ k=1 $: $ d(1,2) = 1 $,
    \item $ k \geq 2 $: $ d(2,2) = 0 $.    
So $ \delta_k = [1, 1, 0, 0, \dots] $.

For $ n \geq 2 $,
 $ \sum_{k=0}^{n-1} \delta_k = 1 + 1 = 2 $;
 $ \sum_{k=0}^{n-1} \delta_{k+1} = \sum_{k=1}^{n} \delta_k = 1 + 0 + \cdots + 0 = 1 $.

So
\[
\sum_{k=0}^{n-1} d(T^{k+1}x, T^{k+1}y) = 1 \leq \alpha \cdot 2
\quad \text{if} \quad \alpha \geq \frac{1}{2} .
\]

\subsubsection*{Case 2: $ x = 0, y = 2 $}

$ k=0 $: $ d(0,2) = 1 $;
$ k=1 $: $ d(1,2) = 1 $;
$ k \geq 2 $: $ d(2,2) = 0 $.

Same as Case 1: sum = 2, shifted sum = 1. So again:
\[
1 \leq 2\alpha \quad \text{if} \quad \alpha \geq 1/2
\]

\subsubsection*{Case 3: $ x = 1, y = 2 $}

$ k=0 $: $ d(1,2) = 1 $,
$ k=1 $: $ d(2,2) = 0 $,
$ k \geq 2 $: $ d(2,2) = 0 $.

So $ \delta_k = [1, 0, 0, \dots] $ .

For $ n \geq 1 $,
$ \sum_{k=0}^{n-1} \delta_k = 1 $;
$ \sum_{k=0}^{n-1} \delta_{k+1} = \sum_{k=1}^{n} \delta_k = 0 $.

So
\[
0 \leq \alpha \cdot 1 \quad \text{for any } \alpha > 0 .
\]

\subsection*{Conclusion of (b)}

For all $ x \ne y $ and $ n \geq 2 $, the worst-case ratio of sums is $ 1/2 $. Therefore, choose for instance $ \alpha = \frac{2}{3} \in (0,1) $ and $ N = 2 $. Then for all $ x, y \in X $, $ n \geq 2 $
\[
\sum_{k=0}^{n-1} d(T^{k+1}x, T^{k+1}y) \leq \frac{2}{3} \sum_{k=0}^{n-1} d(T^k x, T^k y).
\]

Hence, $ T $ is a PA-contraction.

\end{proof}

\begin{remark}
Since every metric space is a b-metric space with coefficient $ s = 1 $, and the counterexample provided lies in such a space, the strict inclusion of Banach contractions in PA-contractions extends to the class of b-metric spaces. Therefore, PA-contractions form a strictly larger class than Banach contractions in both metric and b-metric settings.
\end{remark}


\section{Comparison with Kannan Contractions}
\label{sec:comparison-kannan}

We recall the definition of a Kannan contraction.

\begin{definition}[Kannan contraction]
Let $ (X, d) $ be a b-metric space. A mapping $ T: X \to X $ is a \textbf{Kannan contraction} if there exists $ \beta \in \left(0, \frac{1}{2s}\right) $ such that for all $ x, y \in X $,
\[
d(Tx, Ty) \leq \beta \left[ d(x, Tx) + d(y, Ty) \right].
\]
\label{def:kannan}
\end{definition}

Note that in metric spaces ($ s = 1 $), the condition reduces to $ \beta < \frac{1}{2} $, which is the classical Kannan constant.

We now compare PA-contractions with Kannan contractions in metric spaces. We show that the class of PA-contractions is not contained in the class of Kannan contractions.

In particular, we demonstrate that there exist PA-contractions that are not Kannan contractions.

\begin{example}
We use the mapping $T$ from Example~\ref{ex:discretemetric}.
As shown in Section~\ref{sec:pa-not-ba}, $ T $ is a PA-contraction with $ \alpha = \frac{2}{3} $, $ N = 2 $, but not a Banach contraction.

We now show that $ T $ is also \textbf{not} a Kannan contraction.

Suppose for contradiction that there exists $ \beta < \frac{1}{2} $ such that
\[
d(Tx, Ty) \leq \beta [d(x, Tx) + d(y, Ty)], \quad \forall x, y \in X.
\]

Take $ x = 0 $, $ y = 1 $. Then,
$ d(T(0), T(1)) = d(1,2) = 1 $, 
$ d(0, T(0)) = d(0,1) = 1 $, 
$ d(1, T(1)) = d(1,2) = 1 $.
Right-hand side: $ \beta(1 + 1) = 2\beta < 1 $.

So $ 1 \leq 2\beta < 1 $, contradiction.

Hence, $ T $ is a PA-contraction but not a Kannan or Banach contraction.
\end{example}

\begin{corollary}
The class of PA-contractions is not contained in the class of Kannan contractions.
\end{corollary}

\begin{remark}
The question of whether every Kannan contraction is a PA-contraction remains open. While we have shown that there exist PA-contractions that are not Kannan contractions, we have not established the converse relationship. The structural differences between the definitions — one based on self-distances, the other on orbit averages — suggest that these classes may be incomparable, but a definitive conclusion requires either a proof that all Kannan contractions are PA-contractions or an explicit counterexample of a Kannan contraction that fails the PA condition.
\end{remark}



\section{Conclusion}
\label{sec:conclusion}

In this paper, we have extended the theory of Path-Averaged Contractions (PA-contractions) from metric to b-metric spaces. We proved that every PA-contraction 
on a complete b-metric space has a unique fixed point, provided the contraction constant $ \alpha $ and the b-metric coefficient $ s $ satisfy 
$ s \alpha^{1/N} < 1 $. The result relies on the geometric decay of successive distances and careful application of the generalized triangle inequality.

Furthermore, we clarified the relationship between PA-contractions and two classical classes:
\begin{itemize}
\item Every Banach contraction is a PA-contraction, but not conversely --- hence, PA-contractions strictly generalize Banach contractions.
\item The class of PA-contractions is not contained in the class of Kannan contractions, as demonstrated by an explicit counterexample.
\end{itemize}
These comparisons highlight the distinct nature of path-averaged contraction as a principle that captures cumulative orbital behavior rather than pointwise control.

The relationship between PA-contractions and other classical principles -- such as Chatterjea \cite{Chatterjea1972}, \'Ciri\'c-type mappings \cite{Ciric1974}, and Wardowski-type F-contractions \cite{Wardowski2012,fabiano2022} -- remains open and will be investigated in future work.



\end{document}